\documentclass[11pt,reqno]{amsart}
\usepackage{amsmath, amsthm, amssymb, stmaryrd,color,bm}
\usepackage{enumerate}
\usepackage{url}

\let\OLDthebibliography\thebibliography
\renewcommand\thebibliography[1]{
  \OLDthebibliography{#1}
  \setlength{\parskip}{0pt}
  \setlength{\itemsep}{0pt plus 0.3ex}
}

\usepackage{mathtools}

\usepackage{multirow, array}
\usepackage{placeins}
\usepackage{xcolor}

\usepackage{caption}

\advance \topmargin by -\headheight
\advance \topmargin by -\headsep

\evensidemargin \oddsidemargin
\numberwithin{equation}{section}
\newtheorem{theorem}{Theorem}[section]

\newtheorem{corollary}[theorem]{Corollary}
\newtheorem{lemma}[theorem]{Lemma}

\newtheorem{remark}[theorem]{Remark}

\newtheorem{problem}[theorem]{Problem}

\newtheorem{conjecture}[theorem]{Conjecture}
\newtheorem{proposition}[theorem]{Proposition}

\def\R{\mathbb{R}}
\def\T{\mathrm{T}}
\def\Q{\mathbb{Q}}
\def\Z{\mathbb{Z}}
\def\N{\mathbb{N}}
\def\I{\mathbb{I}}

\def\cH{\mathcal{H}}
\def\cB{\mathcal{B}}
\def\cM{\mathcal{M}}
\def\ccM{\tilde{\mathcal{M}}}

\def\cC{\mathcal{C}}
\def\cN{\mathcal{N}}

\def\cL{\mathcal{L}}
\def\cS{{\mathcal{S}}}

\def\cR{\mathcal{R}}

\newcommand{\gt}{g_{\ve,t}}

\newcommand{\ve}{\varepsilon}

\newcommand{\vv}[1]{{\mathbf{#1}}}
\newcommand{\mv}[1]{{\bm#1}}

\newcommand{\tvv}[1]{{{\tilde{\vv#1}}}}

\newcommand{\rank}{\operatorname{rank}}

\newcommand{\diag}{\operatorname{diag}}

\renewcommand{\tilde}{\widetilde}

\newcommand{\codim}{\operatorname{codim}}
\newcommand{\spec}{\operatorname{spec}}

\newcommand{\Jarnik}{Jarn\' \i k}

\newcommand{\cU}{{\vv U}}
\newcommand{\J}{\mathbf{J}}

\newcommand{\SL}{\operatorname{SL}}

\newcommand{\fM}{\mathfrak{M}}

\newcommand{\zu}{u_1}

\begin{document}


\title{Khintchine's theorem and Diophantine approximation on manifolds}

\author[Victor Beresnevich]{Victor Beresnevich}
\address{Victor Beresnevich, Department of Mathematics, University of York, Heslington, York, YO10 5DD, United Kingdom}
\email{victor.beresnevich@york.ac.uk}

\author[Lei Yang]{Lei Yang}
\address{Lei Yang, College of Mathematics, Sichuan University, Chengdu, Sichuan, 610000, China}
\email{lyang861028@gmail.com}

\subjclass[2000]{11J83, 11J13, 11K60, 11K55}

\keywords{Diophantine approximation on manifolds, Khintchine's theorem, Hausdorff dimension, rational points near manifolds, quantitative non-divergence, spectrum of exponents}

\dedicatory{Dedicated to G.\,A.\,Margulis on the occasion of his 75th birthday}

\begin{abstract}
In this paper we initiate a new approach to studying approximations by rational points to points on smooth submanifolds of $\R^n$.
Our main result is a convergence Khintchine type theorem for arbitrary nondegenerate submanifolds of $\R^n$, which resolves a
longstanding problem in the theory of Diophantine approximation. Furthermore, we refine this result using Hausdorff $s$-measures
and consequently obtain the exact value of the Hausdorff dimension of $\tau$-well approximable points lying on any nondegenerate
submanifold for a range of Diophantine exponents $\tau$ close to $1/n$. Our approach uses geometric and dynamical ideas together
with a new technique of `generic and special parts'. In particular, we establish sharp upper bounds for the number of rational
points of bounded height lying near the generic part of a non-degenerate manifold. In turn, we give an explicit exponentially small bound for the measure of the special part of the manifold.
The latter uses a result of Bernik, Kleinbock and Margulis.
\end{abstract}

\maketitle



\section{Introduction}\label{sec1}

\subsection{Khintchine's theorem and manifolds}\label{KhM}

To begin with, let us recall the notion of $\psi$-approximable points which is convenient for introducing the problems investigated in this paper. Here and elsewhere, $\psi:(0,+\infty)\to(0,1)$ is a function that will be referred to as an {\em approximation function}. We will say that the point $\vv y=(y_1,\dots,y_n)\in\R^n$ is {\em $\psi$-approximable} if the system
\begin{equation}\label{psi-app}
\qquad  \left|y_i-\frac{p_i}{q}\right|<\frac{\psi(q)}{q}\qquad(1\le i\le n)
\end{equation}
holds for infinitely many $(p_1,\dots,p_n,q)\in\Z^n\times\N$. The set of $\psi$-approximable points in $\R^n$ will be denoted by $\cS_n(\psi)$. In the special case of $\psi_{\tau}(q):=q^{-\tau}$ for some $\tau>0$ we will also write $\cS_n(\tau)$ instead of $\cS_n(\psi_{\tau})$. Recall that, by Dirichlet's theorem \cite{Schmidt-1980}, $\cS_n(1/n)=\R^n$. For functions $\psi$ that decay faster that $q^{-1/n}$
Khintchine \cite{Khintchine-1924, Khintchine-1926} discovered the following simple yet powerful criterion for the proximity of rational points to almost all points $\vv y$ in $\R^n$. We state it below in a modern (slightly less restrictive) form, see \cite{MR2508636,MR2576284,MR4125453} for further details and generalisations. In what follows $\cL_n$ denotes Lebesgue measure on $\R^n$ and $\cL_n(X)=\textsc{Full}$ means that the complement to $X\subset\R^n$ has Lebesgue measure zero.

\medskip

\noindent\textbf{Khintchine's theorem}{\bf:} {\em Given any decreasing approximation function $\psi$\footnote{In the original version of Khintchine's theorem, $q \psi(q)$ is assumed to be decreasing, see \cite{MR2508636} for an overview.},
\begin{equation}\label{kchin1}
  \cL_n\big(\cS_n(\psi)\big)=\left\{\begin{array}{cc}
                                  0 & \text{if }\sum_{q=1}^\infty\psi(q)^n<\infty\,, \\[2ex]
                                  \textsc{Full} & \text{if }\sum_{q=1}^\infty\psi(q)^n=\infty\,.
                                \end{array}
  \right.
\end{equation}
}

The convergence case of Khintchine's theorem is a simple application of the Borel-Cantelli lemma based on the trivial count of rational points of bounded height. However, studying the proximity of rational points to points $\vv y$ lying on a submanifold $\cM\subset\R^n$ gives rise to major challenges. Indeed, extending Khintchine's proof to manifolds requires solving the notoriously difficult problem of counting rational points lying close to $\cM$ \cite[\S1.6.1.2]{MR3618787}. This was first observed by Sprind\v zuk in \cite[\S2.6]{Sprindzuk-1979-Metrical-theory}. The main purpose of this paper is to address the following central problem that was initially communicated by Kleinbock and Margulis in their seminal paper on the Baker-Sprin\v zuk conjecture \cite[\S6.3]{Kleinbock-Margulis-98:MR1652916} and later stated in a more general form\footnote{The general form incorporates the supports of so-called friendly measures, which essentially generalise the notion of nondegeneracy from manifolds to fractals. See \cite{Khalil-Luethi} for recent advances on the version of Problem~\ref{p1} for fractals.} by Kleinbock, Lindestrauss and Weiss \cite[Question~10.1]{MR2134453}.

\begin{problem}\label{p1}
Let $\cM\subset\R^n$ be a nondegenerate submanifold. Verify if for any monotonic function $\psi$ almost no/every point of $\cM$ is $\psi$-approximable whenever the series
\begin{equation}\label{sum}
  \sum_{q=1}^\infty\psi(q)^n
\end{equation}
converges/diverges.
\end{problem}

In this paper we use the notion of nondegeneracy introduced in \cite{Kleinbock-Margulis-98:MR1652916}. A map $\vv f  : \cU  \to \R^n$, defined on an open subset $\cU\subset\R^d$, is said to be \emph{$l$--nondegenerate at} $\vv x_0\in \cU$ if $\vv f$  is $l$--times continuously differentiable on a neighborhood of $\vv x_0$ and the partial derivatives
of $\vv f$ at $\vv x_0$ of orders up to $l$ span $\R^n$. The map $\vv f$
is said to be \emph{nondegenerate} at $\vv x_0$ if it is $l$--nondegenerate at $\vv x_0$ for some $l\in\N$. The map $\vv f$ is said to be \emph{nondegenerate} if it is nondegenerate at $\cL_d$\,--almost every point in $\cU$. In turn, the immersed manifold $\cM:=\vv f(\cU)$ is nondegenerate (at $\vv y_0=\vv f(x_0)$) if the immersion $\vv f:\cU\to\R^n$ is nondegenerate (at $\vv x_0$). This readily extends to manifolds $\cM$ that do not posses a global parameterisation via local parameterisations. As is well known, any real connected analytic manifold not contained in a hyperplane of $\R^n$  is nondegenerate \cite{Kleinbock-Margulis-98:MR1652916}. In fact, it is nondegenerate at every point.

The special case $\psi(q)=q^{-\tau}$, $\tau>0$ of Problem~\ref{p1} was posed by Sprind\v zuk for analytic manifolds \cite{Sprindzuk-1980-Achievements} and famously resolved in full by Kleinbock and Margulis \cite{Kleinbock-Margulis-98:MR1652916}. Note
that for these approximation functions the divergence case is trivial thanks to Dirichlet's theorem. In \cite{Kleinbock-03:MR1982150} Kleinbock extended \cite{Kleinbock-Margulis-98:MR1652916} to  affine subspaces of $\R^n$ satisfying certain Diophantine conditions and to submanifolds of such subspaces that are nondegenerate with respect to them. In another direction Kleinbock, Lindestrauss and Weiss \cite{MR2134453} established the analogue of \cite{Kleinbock-Margulis-98:MR1652916} for the supports of friendly measures. We also refer the reader to \cite{Sprindzuk-1979-Metrical-theory} and \cite{BernikDodson-1999} for various preceding results.

For arbitrary monotonic $\psi$, Problem~\ref{p1} turned out to be far more delicate. Its \textbf{divergence case} was settled for $C^3$ planar curves \cite{MR2373145} and then fully resolved for analytic manifolds in arbitrary dimensions \cite{MR2874641}. More recently, the latter was also extended to arbitrary nondegenerate curves \cite{MR4287738}, while for planar curves the nondegeneracy assumption was replaced by weak nondegeneracy in \cite{BZ}\footnote{A curve $\cC\subset\R^2$ is weakly non-degenerate at a point $\vv p\in\cC$ if there is a neighborhood of $\vv p$ that can be written as the uniform limit of a sequence of non-degenerate curves whose curvatures are uniformly bounded away from zero and infinity, see \cite{BZ} for further details.}.

The \textbf{convergence case} of Problem~\ref{p1} for arbitrary $\psi$ is a different story. It was resolved for $n=2$ for all $C^3$ nondegenerate curves in the breakthrough of Vaughan and Velani \cite{Vaughan-Velani-2007}. Later Huang \cite{HuangRPAdv} extended this to weakly nondegenerate planar curves. However, known results in higher dimensions require various additional constrains on the geometry and dimension of manifolds, predominantly as a result of the use of tools based on Fourier analysis. A brief account of known results is as follows.
Bernik \cite{Bernik-77:MR0480402} proved it for the manifolds in $\R^{dk}$ defined as the Cartesian products of $d\ge k\ge2\;$ $C^{k+1}$ nondegenerate curves in $\R^k$.
Dodson, Rynne and Vickers \cite{DodsonRynneVickers-1991a} proved it for the manifolds $\cM$ in $\R^n$ having at least two non-zero principle curvatures of the same sign with respect to every normal direction at $\vv y$ for almost all $\vv y\in\cM$. Note that this geometric condition requires that the dimension $d=\dim\cM$ satisfies the inequality $(d+1)d\ge 2n$. Vaughan, Velani, Zorin and the first named author of this paper \cite{MR3658127} proved the convergence case of Problem~\ref{p1} for $2$--nondegenerate manifolds in $\R^n$ of dimension $d\ge n/2+1$ with $n\ge4$. They also proved it for hypersurfaces in $\R^3$ with Gaussian curvature non-vanishing almost everywhere \cite[Corollary~5]{MR3658127}.
Simmons \cite{MR3809714} further relaxed the conditions of \cite{DodsonRynneVickers-1991a} and \cite{MR3658127} imposed on manifolds, albeit the restrictions on their dimension remain broadly the same and, for instance, rule out curves. In a related development Huang and Liu \cite{JJHJL} proved a Khintchine type theorem for affine subspaces satisfying certain Diophantine conditions.

Thus, the convergence case of Problem~\ref{p1} remains fully open for curves in dimensions $n\ge3$. Indeed, it is open for subclasses of nondegenerate manifolds in $\R^n$ of every dimension $d<n$. Even in the case of hypersurfaces, which are most susceptible to the methods used in preceding papers, the problem is not fully resolved, e.g. it is open for hypersurfaces in $\R^3$ of zero Gaussian curvature. In this paper we contrive no additional hypotheses on nondegenerate manifolds and resolve the convergence case of Problem~\ref{p1} in full. Our main result reads as follows.

\begin{theorem}\label{T1}
Let $n\ge2$, a submanifold $\cM\subset\R^n$ be nondegenerate, $\psi$ be monotonic, and assume \eqref{sum} converges. Then almost all points on $\cM$ are not $\psi$-approximable.
\end{theorem}

Hausdorff measure and Hausdorff dimension are often used to distinguish between sets of Lebesgue measure zero and thus refine the convergence case of Khintchine's theorem. In this paper we make an extra step beyond Theorem~\ref{T1} and establish such refinements. The precise statements are provided in \S\ref{sec1.3}, while \S\ref{sec1.2} contains an overview of preceding results and problems.

\subsection{Rational points near manifolds}\label{RPnM}

The proof of Theorem~\ref{T1} and indeed its generalisation to Hausdorff measures stated in \S\ref{sec1.3} are underpinned by a new result on rational points near nondegenerate manifolds stated in this section, which is of independent interest. For simplicity and without loss of generality we will assume that the manifolds $\cM$ of interest are immersed by maps $\vv f:\cU\to\R^n$, where $\cU\subset\R^d$ denotes an open subset. Furthermore, in view of the Implicit Function Theorem, it is non-restrictive to assume that
\begin{equation}\label{eq3.2}
\vv f (\vv x) = (\vv x,\mv f(\vv x))=(x_1, \dots, x_d, f_{1}(\vv x) , \dots, f_m(\vv x))\,,
\end{equation}
where $d=\dim\cM$ and $m=\codim\cM$, $\vv x=(x_1,\dots,x_d)\in\cU$ and
\begin{equation}\label{e3.2}
\mv f=(f_1,\dots,f_{m}):\cU\to\R^{m}\,.
\end{equation}
Since we are interested in nondegenerate manifolds, the maps $\vv f$ must necessarily be $C^2$. Furthermore,
when proving Theorem \ref{T1}, we can deal with the manifold locally. Therefore,
without loss of generality we can assume that there is a constant $M\ge1$ such that
\begin{equation} \label{lemma_translation_M}
\max_{1\le k\le m}\;\max_{1\le i,j\le d}\;\sup_{\vv x\in \cU}\;\max\left\{\left|\frac{\partial f_k(\vv x)}{\partial x_i}\right|,\;
\left|\frac{\partial^2 f_k(\vv x)}{\partial x_i\partial x_j}\right|\right\}\leq M\,.
\end{equation}

Given $t>0$, $0<\ve<1$ and $\Delta\subset\R^d$, let
$$
\cR(\Delta;\ve,t)=\left\{(\vv p,q) \in \Z^{n+1}:0< q < e^t\text{ and }
  \inf_{\vv x \in \Delta\cap\cU}\left\| \vv f(\vv x) -\frac{\vv p}{q} \right\|_{\infty} <  \frac{\ve}{e^t}
\right\}
$$
and let
$$
N(\Delta;\ve,t)=\#\cR(\Delta;\ve,t)\,.
$$
Thus, $N(\Delta;\ve,t)$ counts the rational points $\vv p/q$ (not necessarily written in the lowest terms) of denominator $0<q< e^t$ lying $\ve e^{-t}$--close to $\vv f(\Delta\cap\cU)\subset\cM$.

Counting rational points on manifolds is usually geared towards establishing estimates of the form
\begin{equation}\label{e}
  N(\cU;\ve,e^t)\ll \ve^{m}e^{(d+1)t}+E(\cU;\ve,t)\,,
\end{equation}
where $E(\cU;\ve,t)$ is an error term. In general $E(\cU;\ve,t)$ cannot be made smaller than $e^{dt}$ for all
nondegenerate manifolds since a $d$--dimensional nondegenerate manifold may contain a $(d-1)$--dimensional
rational subspace. Furthermore, a nondegenerate manifold may accumulate abnormally high number of rational
points around points where it has a very high `contact' with its tangent $d$--dimensional plane, for example,
when this tangent plane is rational. In this case even the estimate $E(\cU; \ve,t)\ll e^{(d+1-\delta)t}$ may not be achievable for any $\delta>0$. Hence, establishing \eqref{e} with a useful error term requires imposing conditions beyond nondegeneracy. Worse still, major limitations on the dimension $d$ and/or $\ve$ arise from the tools that are currently in use. In fact, the theory is only reasonably complete for curves is $\R^2$
thanks to the breakthrough of Vaughan and Velani \cite{Vaughan-Velani-2007} who proved \eqref{e} with $E(\cU;\ve,t)= O( e^{t(1+\delta)})$ and arbitrary $\delta>0$ for any compact $C^3$ curve in $\R^2$ of nonzero curvature, and thus obtained the best possible strengthening of Huxley's earlier result \cite{MR1310631}.
The complementary lower bound was found in \cite{MR2373145} and various improvements for planar curves can be found in \cite{BZ, MR3263942, MR3630723, HuangRPAdv, MR4093379}. Results in higher dimensions are limited to manifolds with various additional hypothesis that we already discussed in \S\ref{KhM} and can be found in \cite{MR3658127, MR3809714, JJHJL}. More recently Huang \cite{MR4132580} obtained essentially the best possible bound on the error term in \eqref{e} for a class of hypersurfaces in $\R^n$ with non-zero Gaussian curvature. A further generalisation of Huang was found by Schindler and Yamagishi \cite{MR4404031}. It is commonly believed that the biggest challenge is establishing \eqref{e} is posed by the case of curves in $\R^n$ for which we have virtually no results. In this context it is worth mentioning the recent work of Huang \cite{MR3985129} on rational points near nondegenerate curves in $\R^3$ of fixed denominator and whose main result implies \eqref{e} with $E(\cU;\ve,t)=O(t^{4/5}e^{8t/5})$. However, with reference to this bound, the main term in \eqref{e} becomes dominant only when $\ve\ge e^{-t/5}$ and so it cannot be used to resolve Problem~\ref{p1} for curves in $\R^3$, which requires understanding rational points lying much closer to the manifolds in question, namely $\ve=o(e^{-t/3})$.

In this paper we deal with all nondegenerate manifolds including nondegenerate curves by introducing a new approach,
which involves splitting a manifold into a `generic part' and a `special part' using diagonal actions on the space of lattices. In short, we establish a sharp upper bound for the number of rational points lying near the generic part, which agrees with the main term in \eqref{e}, see \eqref{eq1.5} below. Regarding the special part, we establish explicit bounds on the size of the special part which decay exponentially and uniformly for $e^{-\frac{3t}{2n}+\delta}\le\ve<1$ as $t\to\infty$, where $\delta>0$ is arbitrary, see \eqref{eq1.4} below.
The key new idea is to apply the so-called quantitative
non-divergence estimate on the space of
lattices \cite[Theorem~6.2]{Bernik-Kleinbock-Margulis-01:MR1829381} in order to demonstrate that the size of the special part is small, and to use certain tools from homogeneous dynamics and
the geometry of numbers to count rational points near the generic part. The previous results that we discussed above have been using tools from
analytic number theory (a version of the circle method) relying on Fourier analysis, and Huxley \cite{MR1310631} used a version of the determinate method of Swinnerton-Dyer. Our main result on rational points is as follows.

\begin{theorem}\label{t1.3}
Suppose $\cU\subset\R^d$ is open, $\vv f:\cU\to\R^n$ be a $C^2$ map satisfying \eqref{eq3.2} and \eqref{lemma_translation_M}. Then for any $0<\ve<1$ and every $t>0$ there is a subset $\fM(\ve,t)\subset \cU$, which can be written as a union of balls in $\cU$ of radius $\ve e^{-t/2}$ of intersection multiplicity $\le N_d$, where $N_d$ is the Besicovitch constant, satisfying the following properties. For every $\vv x_0\in\cU$ such that $\vv f$ is $l$--nondegenerate at $\vv x_0$ there is a ball $B_0\subset\cU$ centred at $\vv x_0$ and constants $K_0,t_0>0$, depending on $B_0$ and $\vv f$ only, such that
\begin{equation}\label{eq1.4}
\cL_d\big(\fM(\ve, t)\cap B_0\big) \le K_0 \left(\ve^{n}e^{\frac{3t}{2}}\right)^{-\frac{1}{d(2l-1)(n+1)}}\quad\text{for }t\ge t_0
\end{equation}
and for every ball $B\subset\cU$, for all sufficiently large $t$ we have that
\begin{equation}\label{eq1.5}
N\Big(B\setminus\fM(\ve,t);\ve,t\Big)\le K_1\ve^{m}e^{(d+1)t} \cL_d(B)\,,
\end{equation}
where $K_1$ depends on $n$ and $\vv f$ only.
\end{theorem}

We now demonstrate how Theorem~\ref{t1.3} is used to resolve Problem~\ref{p1}.

\subsection{Proof of Theorem~\ref{T1} modulo Theorem~\ref{t1.3}}
To begin with we give the following two auxiliary statements.

\begin{lemma}\label{l1.4}
If $\vv f(\vv x)\in\cS_n(\psi)$ then there are infinitely many $t\in\N$ such that
\begin{equation}\label{eq1.9}
\left\| \vv f(\vv x) -\frac{\vv p}{q} \right\|_{\infty}<\frac{\psi(e^{t-1})}{e^{t-1}}
\end{equation}
for some $(\vv p,q)\in\Z^{n+1}$ with $e^{t-1}\le q<e^t$.
\end{lemma}

\begin{proof}
If $\vv f(\vv x)\in\cS_n(\psi)$ then \eqref{psi-app} holds for infinitely many
$(\vv p,q)\in\Z^{n+1}$ with arbitrarily large $q>0$. For each $q$ we can define the corresponding $t\in\N$ from the inequalities $e^{t-1}\le q<e^t$. There are infinitely many $t\in\N$ arising this way since $q$ is unbounded. Finally \eqref{eq1.9} follows from \eqref{psi-app} since $\psi$ is monotonically decreasing.
\end{proof}

\begin{lemma}\label{l1.5}
Let $\psi:\R\to\R$ be monotonic. Then
$$
\sum_{q=1}^\infty\psi(q)^n<\infty\qquad\iff\qquad \sum_{t=1}^\infty\psi(e^t)^ne^t<\infty
$$
\end{lemma}

Lemma~\ref{l1.5} is a version of the Cauchy condensation test.

\begin{proof}[Proof of Theorem~\ref{T1} modulo Theorem~\ref{t1.3}]
Without loss of generality we consider $\cM$ of the form $\vv f(\cU)$, where $\vv f:\cU\to\R^n$ is a nondegenerate immersion on an open subset $\cU\subset\R^d$. Since $\vv f$ is nondegenerate, for almost every $\vv x_0\in\cU$ the map $\vv f$ is nondegenerate at $\vv x_0$. Hence, without loss of generality, it suffices to prove that
$$
\cL_d\big(\big\{\vv x\in B_0:\vv f(\vv x)\in\cS_n(\psi)\big\}\big)=0\qquad\text{if \eqref{sum} converges and $\psi$ is monotonic}
$$
for a sufficiently small ball $B_0$ centred at $\vv x_0\in\cU$ where $\vv f$ is $l$--nondegenerate at $\vv x_0$ for some $l\in\N$. Fix $\vv x_0$ and take $B_0$ as in Theorem~\ref{T1}.

Without loss of generality we will assume that $\psi(q)\ge q^{-5/4n}$ for all $q>0$, as otherwise we can replace $\psi$ with $\max\{\psi(q),q^{-5/4n}\}$.
By Lemma~\ref{l1.4}, for any $T\ge1$
\begin{align}
\nonumber&\big\{\vv x\in B_0:\vv f(\vv x)\in\cS_n(\psi)\big\}\subset\;\bigcup_{t\ge T}\underbrace{\left(\fM(e\psi(e^{t-1}), t)\cap B_0\right)}_{A_t}\;\;\bigcup\\[0ex]
&\bigcup_{t\ge T}\underbrace{\bigcup_{(\vv p,q)\in\cR(B_0\setminus \fM(e\psi(e^{t-1}), t);e\psi(e^{t-1}),t)}\left\{\vv x\in B_0: \left\| \vv x -\frac{\vv p'}{q} \right\|_{\infty}<\frac{\psi(e^{t-1})}{e^{t-1}}
\right\}}_{B_t}\,.\label{eq1.10}
\end{align}
By Theorem~\ref{t1.3} and the assumption $\psi(q)\ge q^{-5/4n}$, we get that
$$
\cL_d\left(A_t\right)\ll \left(e^{(t-1)/4}\right)^{-\frac{1}{d(2l-1)(n+1)}}
$$
and
$$
\cL_d\left(B_t\right)\ll \psi(e^{t-1})^{m}e^{(d+1)(t-1)} \;\cdot\;\left(\frac{\psi(e^{t-1})}{e^{t-1}}\right)^d\;=\;\psi(e^{t-1})^ne^{t-1}\,.
$$
Hence, by Lemma~\ref{l1.5}, we get that
\begin{align*}
\cL_d\left(\big\{\vv x\in B_0:\vv f(\vv x)\in\cS_n(\psi)\big\}\right)\ll\;\sum_{t\ge T}
\left(e^{(t-1)/4}\right)^{-\frac{1}{d(2l-1)(n+1)}}\;\;+\;\;\sum_{t\ge T}\psi(e^{t-1})^ne^{t-1}
\end{align*}
which tends to $0$ as $T\to\infty$ since the series above are convergent. Therefore, $\cL_d\left(\big\{\vv x\in B_0:\vv f(\vv x)\in\cS_n(\psi)\big\}\right)=0$ and the proof is complete.
\end{proof}

\medskip

\section{Generalisations to Hausdorff measure and dimension}\label{sec2}

\subsection{Problems and known results}\label{sec1.2}
To begin with, let us recall two classical results in this area due to Jarn\'ik and Besicovitch, which represent the Hausdorff dimension and Hausdorff measure refinements of Khintchine's theorem.

\smallskip

\noindent\textbf{The Jarn\'ik-Besicovitch theorem} \cite{Besicovich-1934, Jarnik-1929}{\bf:} {\em Let $\tau\ge 1/n$. Then
\begin{equation}\label{JB}
  \dim \cS_n(\tau)=\frac{n+1}{\tau+1}\,.
\end{equation}
}

\noindent\textbf{Jarn\'ik's theorem\footnote{The original statement of Jarn\'ik's theorem had additional constrains, see \cite{MR2508636} for details.}} \cite{Jarnik-1931}{\bf:} {\em Given any monotonic function $\psi$ and $0<s<n$,
\begin{equation}\label{Jar}
  \cH^s\big(\cS_n(\psi)\big)=\left\{\begin{array}{cc}
                                  0 & \text{if }\sum_{q=1}^\infty \
 q^n\Big(\frac{\psi(q)}q\Big)^{s}<\infty\,, \\[2ex]
                                  \infty & \text{if }\sum_{q=1}^\infty \
 q^n\Big(\frac{\psi(q)}q\Big)^{s}=\infty\,.
                                \end{array}
  \right.
\end{equation}
}

In the above and elsewhere `$\dim$' denotes the Hausdorff dimension and $\cH^s$ denotes the $s$--dimensional Hausdorff measure.
The following general problem aims at refining the Lebesgue measure theory of $\psi$-approximable points on manifolds.
It is geared towards establishing the analogues of the theorems of Jarn\'ik and Besicovitch,
and it incorporates Khintchine type theorems for manifolds as the special case $s=\dim\cM$.

\begin{problem}\label{p2}
Given a smooth submanifold $\cM\subset\R^n$, determine the Hausdorff dimension $s$ of the set $\cS_n(\psi)\cap\cM$ and furthermore determine the $s$-dimensional Hausdorff measure of the set $\cS_n(\psi)\cap\cM$.
\end{problem}

As before, our interest in Problem~\ref{p2} will be focused on nondegenerate manifolds. It is well known that for approximation functions $\psi$ that decay relatively fast the problem cannot have the same answer for all such manifolds, even if the manifolds are nondegenerate at every point. This is easily illustrated by the following example, whose details can be found in \cite{MR2874641}. Let $\cC_r$ be the circle in $\R^2$ defined by the equation $x^2+y^2=r$. Then
\begin{equation}\label{circle}
\dim\cS_2(\tau)\cap\cC_1=\frac{1}{\tau+1}\qquad\text{while}\qquad
\dim\cS_2(\tau)\cap\cC_3=0\qquad\text{for all $\tau>1$}.
\end{equation}
This naturally leads to the following

\begin{problem}[Dimension Problem]\label{p1.4}
Let $1\le d<n$ be integers and $\ccM$ be a class of submanifolds $\cM\subset\R^n$ of dimension $d$. Find the maximal value $\tau(\ccM)$ such that
\begin{equation}\label{dimconj}
 \dim\cS_n(\tau)\cap\cM=\frac{n+1}{\tau+1}-\codim\cM\quad\text{ whenever }
1/n\le\tau<\tau(\ccM)
\end{equation}
for every manifold $\cM\in \ccM$. In particular, find $\tau_{n,d}:=\tau(\ccM_{n,d})$ for the class  $\ccM_{n,d}$ of
manifolds in $\R^n$ of dimension $d$ which are nondegenerate at every
point\footnote{The hypothesis of nondegeneracy can be asked everywhere except on a set of dimension $\le\dim\cS_n(\tau)\cap\cM$. However, this relaxation will not make the problem more general.}.
\end{problem}

Formula \eqref{dimconj} for the dimension is informed by the volume-based expectation for the number of rational points lying close to $\cM$, see \cite{MR2874641} and \cite[\S1.6.2]{MR3618787}. The Dimension Problem was resolved for nondegenerate planar curves in \cite{MR2373145} and \cite{BZ} on establishing that $\tau_{2,1}=1$. Furthermore, a \Jarnik{} type theorem was established in \cite{MR2373145} and \cite{Vaughan-Velani-2007} regarding the
$s$-dimensional Hausdorff measure of $\cS_2(\psi)\cap\cC$ for $C^3$ nondegenerate planar curves $\cC$.

\begin{theorem}[See \cite{MR2373145,Vaughan-Velani-2007}]\label{t1.5}
Given any monotonic approximation function $\psi$, any $s\in(\tfrac12,1)$ and any $C^3$ planar curve $\cC$ nondegenerate at every point, we have that
\begin{equation}\label{jar1}
  \cH^s\big(\cS_2(\psi)\cap\cC\big)=\left\{\begin{array}{cc}
                                  0 & \text{if }\sum_{q=1}^\infty q^{1-s}\psi(q)^{1+s}<\infty\,, \\[2ex]
                                  \infty & \text{if }\sum_{q=1}^\infty q^{1-s}\psi(q)^{1+s}=\infty\,.
                                \end{array}
  \right.
\end{equation}
\end{theorem}

The $C^3$ hypothesis was removed from Theorem~\ref{t1.5} in the case of divergence \cite{BZ} and for a subrange of $s$ in the case of convergence \cite{HuangRPAdv}, where Theorem~\ref{t1.5} was extended to weakly nondegenerate curves.

In higher dimensions there are various speculations as to what $\tau_{n,d}$ might be. Let us first discuss the manifolds of dimension $d>1$.  Consider the nondegenerate manifold $\cM$ in $\R^n$ immersed by the map
\begin{equation}\label{eq2.6}
(x_1,\dots,x_d)\mapsto(x_1,\dots,x_d,x_d^2,\dots,x_d^{n+1-d}).
\end{equation}
Then $\cM$ contains $\R^{d-1}\times\{\vv 0\}$ and so $\dim\cS_n(\tau)\cap\cM\ge \dim\cS_{d-1}(\tau)$. Therefore, by the Jarn\'ik-Besicovitch theorem, we have that
$$
 \dim\cS_n(\tau)\cap\cM\ge \frac{d}{\tau+1}>\frac{n+1}{\tau+1}-\codim\cM\quad\text{ whenever }\tau> 1/(n-d)\,.
$$
This means that $\tau_{n,d}\le 1/(n-d)$ for $d>1$. Any improvement to this hard bound on $\tau_{n,d}$ would require restricting $\cM$ to a smaller subclass of manifolds. Nevertheless,  in all likelihood within the class $\ccM_{n,d}$ of nondegenerate manifolds defined within Problem~\ref{p1.4} this upper bound is exact. We state this formally now as a conjecture.

\begin{conjecture}\label{conj1}
Let $1<d<n$. Then $\tau_{n,d}=\frac{1}{n-d}$.
\end{conjecture}

The following lower bound towards Conjecture~\ref{conj1}
was established in \cite{MR3731303}:
\begin{equation}\label{dimconjL}
 \dim\cS_n(\tau)\cap\cM\ge \frac{n+1}{\tau+1}-\codim\cM\quad\text{ whenever }
\frac1n\le\tau<\frac{1}{n-d}\,,
\end{equation}
which is valid literally for every $C^2$ submanifold $\cM\subset\R^n$ of every dimension $1\le d<n$. In particular, it does not require nondegeneracy or any other constrain on $\cM$. Furthermore, in the case of analytic nondegenerate submanifolds of $\R^n$ the following more subtle Hausdorff measure version of \eqref{dimconjL} for generic $\psi$ was obtained in \cite{MR2874641} generalising the divergence part of Theorem~\ref{t1.5}.

\begin{theorem}[Theorem~2.5 in \cite{MR2874641}]\label{t2.5}
For every analytic nondegenerate submanifold $\cM$ of $\R^n$ of dimension $d$ and codimension $m=n-d$, any monotonic $\psi$ such that $q\psi(q)^m\to\infty$ as $q\to\infty$ and any $s\in(\frac{md}{m+1},d)$ we have that
\begin{equation}\label{e1.10}
 \cH^s(\cS_n(\psi)\cap\cM)=\infty
\end{equation}
whenever the series
\begin{equation}\label{conv2}
  \sum_{q=1}^\infty \
 q^n\Big(\frac{\psi(q)}q\Big)^{s+m}
\end{equation}
diverges.
\end{theorem}

The remaining problem in establishing Conjecture~\ref{conj1} is to get the upper bound for the dimension. Partial progress was made in \cite{MR3658127, MR3985129, MR4132580, JJHJL, MR4404031, MR3809714} as a consequence of results on counting rational points, see \S\ref{RPnM}. However, as with Problem~\ref{p1}, Problem~\ref{p1.4} remains open for curves in dimensions $n\ge3$ and subclasses of nondegenerate manifolds in $\R^n$ of every dimension $d<n$.

Non-degenerate curves are of special interest for various reasons. First of all curves cannot contain rational subspaces and so example \eqref{eq2.6} is not applicable to them. Curves can be used to analyse manifolds of higher dimensions using fibering techniques. In fact, Theorem~\ref{t2.5} and consequently the lower bound \eqref{dimconjL} hold in the follwoing stronger form for nondegenerate curves.

\begin{theorem}[See {\cite[Theorem~7.2]{MR2874641}} and {\cite{MR4287738}}]\label{t2.6}
For every curve $\cC$ in $\R^n$ nondegenerate at every point, any monotonic $\psi$ such that $q\psi(q)^{(2n-1)/3}\to\infty$ as $q\to\infty$ and any $s\in(\frac{1}{2},1)$ we have that
\begin{equation}\label{e1.11}
 \cH^s(\cS_n(\psi)\cap\cC)=\infty
\end{equation}
whenever the series \eqref{conv2} with $m=n-1$ diverges. Consequently,
\begin{equation}\label{dimconjL2}
 \dim\cS_n(\tau)\cap\cC\ge \frac{n+1}{\tau+1}-(n-1)\quad\text{ whenever }
\frac1n\le\tau<\frac{3}{2n-1}\,.
\end{equation}
\end{theorem}

\begin{conjecture}[Curves]
For every $n\ge2$ we have that $\tau_{n,1}=\frac{3}{2n-1}$.
\end{conjecture}

\bigskip

\subsection{New results on Hausdorff measure and dimension}\label{sec1.3}

Here we provide generalisations of Theorem~\ref{T1} to $s$-dimensional Hausdorff measures and Hausdorff dimension, which thus contribute to resolving the problems surveyed in \S\ref{sec1.2}. The following is our key outcome on Hausdorff measures.

\begin{theorem}\label{T2}
Let $n\ge2$ be an integer, $s>0$ and $\cM$ be a submanifold of\/ $\R^n$ such that
\begin{equation}\label{cond1}
  \cH^s\left(\{\vv y\in\cM: \cM\text{ is not $l$--nondegenerate at $\vv y$}\}\right)=0.
\end{equation}
Let $d=\dim\cM$, $m=\codim\cM$ and $\psi$ be a monotonic approximation function such that the series \eqref{conv2}
converges and
\begin{equation}\label{conv3}
  \sum_{t=1}^\infty \ \left(\frac{\psi(e^t)}{e^{\frac t2}}\right)^{s-d}(\psi(e^t)^ne^{\frac{3t}{2}})^{-\alpha}<\infty\qquad\text{where }\alpha:=\tfrac{1}{d(2l-1)(n+1)}\,.
\end{equation}
Then
\begin{equation}\label{e:012}
 \cH^s(\cS_n(\psi)\cap\cM)=0\,.
\end{equation}
\end{theorem}

And the following statement is our key result on Hausdorff dimension.

\begin{corollary}\label{cor2.10}
Let $n\ge2$ be an integer, $\cM$ be a submanifold of $\R^n$ of dimension $d$, which is $l$--nondegenerate everywhere except possibly on a set of Hausdorff dimension $\le \frac{n+1}{\tau+1}-\codim\cM$. Let $\tau\ge 1/n$ satisfy
\begin{equation}\label{eq2.14}
  \frac{n\tau-1}{\tau+1}\le \frac{\alpha(3-2n\tau)}{2\tau+1}\,,
\end{equation}
where $\alpha$ is the same as in \eqref{conv3}. Then
\begin{equation}\label{mainthm2}
  \dim(\cM\cap\cS_n(\tau))=\frac{n+1}{\tau+1}-\codim\cM\,.
\end{equation}
\end{corollary}

Similarly to Theorem~\ref{T1} the proof of Theorem~\ref{T2} is a rather simple consequence of our main result on rational points. We make no delay in showing its details.

\begin{proof}[Proof of Theorem~\ref{T2} modulo Theorem~\ref{t1.3}]
First of all note that since $\psi:\R\to\R$ is monotonic, by Cauchy's condensation test, we have that
\begin{equation}\label{eq2.17}
  \sum_{q=1}^\infty \
 q^n\Big(\frac{\psi(q)}q\Big)^{s+m}<\infty\qquad\iff\qquad \sum_{t=1}^\infty \
 e^{(n+1)t}\Big(\frac{\psi(e^t)}e^t\Big)^{s+m}<\infty\,.
\end{equation}
As before, without loss of generality we consider $\cM$ of the form $\vv f(\cU)$, where $\vv f:\cU\to\R^n$ is a nondegenerate immersion of an open subset $\cU\subset\R^d$. By \eqref{cond1}, it suffices to prove that
$$
\cH^s\big(\big\{\vv x\in B_0:\vv f(\vv x)\in\cS_n(\psi)\big\}\big)=0
$$
whenever $\psi$ is monotonic, \eqref{sum} converges and \eqref{conv3} holds, where $B_0$ is a sufficiently small ball centred at $\vv x_0\in\cU$ where $\vv f$ is $l$--nondegenerate at $\vv x_0$. Fix $B_0$ as in Theorem~\ref{T1}.
By \eqref{l1.4}, for any $T\ge1$ we have inclusion \eqref{eq1.10}.
By Theorem~\ref{t1.3} with $\ve=e\psi(e^{t-1})$, the set $A_t$ can be covered by
$$
\ll \left(\psi(e^{t-1}) e^{-(t-1)/2}\right)^{-d}\left(\psi(e^{t-1})^{n}e^{\frac{3(t-1)}{2}}\right)^{-\frac{1}{d(2l-1)(n+1)}}
$$
balls of radius $\psi(e^{t-1}) e^{-(t-1)/2}$. Furthermore, by Theorem~\ref{t1.3}, we also have that the set $B_t$ is the union of $\ll\psi(e^{t-1})^{m}e^{(d+1)(t-1)}$ balls of radius $\ll \frac{\psi(e^{t-1})}{e^{t-1}}$. Hence, by the definition of $s$-dimensional Hausdorff measure, we get that
$$
\cH^s\left(\big\{\vv x\in B_0:\vv f(\vv x)\in\cS_n(\psi)\big\}\right)\ll \sum_{t\ge T}\psi(e^{t-1})^{m}e^{(d+1)(t-1)} \;\cdot\;\left(\frac{\psi(e^{t-1})}{e^{t-1}}\right)^s+
$$
\begin{equation}\label{eq2.15}
+\sum_{t\ge T}\left(\psi(e^{t-1}) e^{-(t-1)/2}\right)^{s-d}\left(\psi(e^{t-1})^{n}e^{\frac{3(t-1)}{2}}\right)^{-\frac{1}{d(2l-1)(n+1)}}\,.
\end{equation}
The first sum equals
$$
\sum_{t\ge T}e^{(n+1)(t-1)} \left(\frac{\psi(e^{t-1})}{e^{t-1}}\right)^{s+m}
$$
and, by \eqref{eq2.17}, tends to zero. The second sum in \eqref{eq2.15} also tends to zero as a consequence of \eqref{conv3}. Hence, $\cH^s\left(\big\{\vv x\in B_0:\vv f(\vv x)\in\cS_n(\psi)\big\}\right)=0$ and the proof is complete.
\end{proof}

\begin{proof}[Proof of Corollary~\ref{cor2.10}]
Let $L(\tau)$ denote the left hand side of \eqref{eq2.14} and $R(\tau)$ denote the right hand side of
\eqref{eq2.14}. First observe that $L(\tau)$ is increasing and $R(\tau)$ is decreasing. Next, note that $L(1/n)=0$ while $R(1/n)>0$. Also, observe that $L(1/(n-1))=1/n$ while $R(1/(n-1))<\alpha\le 1/(n+1)<L(1/(n-1))$. Hence, the set of solutions to \eqref{eq2.14} is a closed interval $I_{d,l,n}\subset[\frac1n,\frac1{n-1})$. In particular, for any $\tau\in I_{d,l,n}$ estimate \eqref{dimconjL} is applicable and therefore, to prove \eqref{mainthm2}, we only need to prove the complementary upper bound. To this end, let $s>\frac{n+1}{\tau+1}-\codim\cM$ and $\psi(q)=q^{-\tau}$. Then it is readily seen that \eqref{conv2} is convergent. Furthermore, by \eqref{eq2.14}, one easily verifies condition \eqref{conv3}. Condition \eqref{cond1} is also satisfied since $\cM$ is $l$--nondegenerate everywhere except possibly on a set of Hausdorff dimension $<s$. Hence, Theorem~\ref{T2} is applicable and we conclude that $\cH^s(\cS_n(\tau)\cap\cM)=0$. By definition, it means that $\dim (\cS_n(\tau)\cap\cM)\le s$. Since $s>\frac{n+1}{\tau+1}-\codim\cM$ is arbitrary we obtain the require upper bound and complete the proof.
\end{proof}

\begin{remark}\rm
It is not difficult to see that the monotonicity of $\psi$ was only used to apply the Cauchy condensation test to establish \eqref{eq2.17} and to replace $\psi(q)/q$ with $\psi(e^{t-1})/e^{t-1}$ in Diophantine inequalities. The requirement that $\psi$ is monotonic within Theorems~\ref{T1} and \ref{T2} can therefore be replaced with a weaker assumption. For instance, one can replace the monotonicity of $\psi$ with the following requirement: there exist a constant $C>0$ such that
$$
\psi(q)\le C \psi(e^{t-1})\qquad\text{for $e^{t-1}\le q< e^t$}\,.
$$
In fact the use of the sequence $e^t$ is not critical and it can be replaced by any other sequence $s_t>0$ such that $1<\liminf_{t\to\infty}s_t/s_{t-1}\le \limsup_{t\to\infty}s_t/s_{t-1}<\infty$.
\end{remark}

\subsection{Spectrum of Diophantine exponents}\label{sec2.3}

Now let us describe the implications of our results for a problem of Bugeaud and Laurent regarding the spectrum of the following Diophantine exponent introduced in \cite{MR2149403}.
Given $x\in\R$, let
$$
\lambda_n(x):=\sup\left\{\tau>0:(x,x^2,\dots,x^n)\in \cS_n(\tau)\right\}
$$
be the exponent of simultaneous rational approximations to $n$ consecutive powers of a real number $x$. By Dirichlet's theorem, we have that $\lambda_n(x)\in[\frac1n,+\infty]$ for any $x\in\R$. The spectrum of $\lambda_n$ is defined as
$$
\spec(\lambda_n):=\lambda_n(\R\setminus\Q)=\left\{\lambda\in\left[\tfrac1n,+\infty\right]:\exists\;x\in\R\setminus\Q\;\text{with}\;\lambda_n(x)=\lambda\right\}\,.
$$
In 2007 Bugeaud and Laurent posed the following problem.

\begin{problem}[Bugeaud-Laurent {\cite[Prob.~5.5]{MR2349650}}]\label{p2.10}
Is\; $\spec(\lambda_n)=[\frac1n,+\infty]$\;?
\end{problem}

The following more subtle version of this problem was
later raised in \cite{MR2791654}:

\begin{problem}[Bugeaud {\cite[Problem~3.5]{MR2791654}}]\label{p2.11}
For every $\lambda\ge\tfrac1n$ determine $\dim\{x\in\R:\lambda_n(x)=\lambda\}$ and $\dim\{x\in\R:\lambda_n(x)\ge\lambda\}$.
\end{problem}

To begin with, note that, by Sprind\v zuk's theorem \cite{Sprindzuk-1969-Mahler-problem}, $\lambda_n(x)=\tfrac1n$ for almost all $x\in\R$. In particular, $\tfrac1n\in\spec(\lambda_n)$ for every $n$.
For $n=1$ Problem~\ref{p2.10} is relatively simple and can be solved, for instance, using continued fractions, while the answer to Problem~\ref{p2.11} is provided by the \Jarnik-Besicovitch theorem stated at the start of \S\ref{sec1.2}. For $n=2$, Problem~\ref{p2.11}, and consequently Problem~\ref{p2.10}, was solved in \cite{MR2373145} and \cite{MR2604984}.
In turn, Bugeaud \cite{MR2791654} showed that $[1,+\infty]\subset\spec(\lambda_n)$ for any $n$ using explicit examples, while Schleischitz \cite{MR3450571} resolved Problem~\ref{p2.11} for $\lambda>1$. The most significant challenge within Problems~\ref{p2.10} and \ref{p2.11} is posed by the values of $\lambda$ in the spectrum of $\lambda_n$ which are $<1$. The first step in this direction was made by Schleischitz \cite{MR3724166} who proved that $\spec(\lambda_3)$ contains points $<1$. Most recently, Badziahin and Bugeaud \cite{MR4089038} made a major achievement by showing that
$$
\left[\tfrac{n+4}{3n},+\infty\right]\subset\spec(\lambda_n)\quad\text{for every $n\ge3$}
$$
and resolving Problem~\ref{p2.11} for $\lambda\ge\frac{n+4}{3n}$. Corollary~\ref{cor2.10} of our paper makes a first step in closing the gap in the spectrum of $\lambda_n$ from the other end, namely for the values $\lambda$ close to the Dirichlet exponent $1/n$. To produce an explicit statement we now specialise Corollary~\ref{cor2.10} to curves. First of all we state and prove the following proposition which allows us to fix the nondegeneracy parameter $l$.

\begin{proposition}\label{prop2.13}
Let $\vv f:\cU\to\R^n$ be $l$-nondegenerate at $x_0\in\cU$, where $\cU$ is an interval in $\R$. Then there is an interval $B_0$ centred at $x_0$ and a countable subset $S\subset B_0$ such that $\vv f$ is $n$-nondegenerate at every point $x\in B_0\setminus S$.
\end{proposition}

This proposition is a standard exercise in analysis relying on the following

\begin{lemma}\label{l2.14}
If $\varphi:\cU\to\R$ is a $C^1$ function on an interval $\cU$ and $\cN(\varphi):=\{x\in\cU:\varphi(x)=0\}$ then $\cN(\varphi)\setminus\cN(\varphi')$ consists on isolated points.
\end{lemma}

\begin{proof}
If $x_0\in \cN(\varphi)\setminus\cN(\varphi')$ is a limit point of $\cN(\varphi)$ then there is a sequence $x_k\in\cU\cap \cN(\varphi)\setminus\{x_0\}$ converging to $x_0$. By the Mean Value Theorem, $\varphi'(\tilde x_k)(x_k-x_0)=\varphi(x_k)-\varphi(x_0)=0$, where $\tilde x_k$ between $x_k$ and $x_0$. Thus, $\varphi'(\tilde x_k)=0$. Letting $k\to\infty$ and using the continuity of $\varphi'$ gives $\varphi'(x_0)=\varphi'(\lim_{k\to\infty}\tilde x_k)=\lim_{k\to\infty}\varphi'(\tilde x_k)=0$. However, $x_0\not\in \cN(\varphi')$. Thus $x_0$ cannot be a limit point of $\cN(\varphi)$.
\end{proof}

\begin{proof}[Proof of Proposition~\ref{prop2.13}]
Since $\vv f$ is $l$--nondegenerate at $x_0$, we have that $\rank\{\vv f^{(i)}(x_0):1\le i\le l\}=n$. Since $\vv f$ is $C^l$ there is an interval centred at $x_0$ such that $\rank\{\vv f^{(i)}(x):1\le i\le l\}=n$ for all $x\in B_0$. If $l=n$ there the statement is obvious. Thus we will assume that $l>n$. Let $\varphi(x):=\det(f^{(i)}_j(x))_{1\le i,j\le n}$ be the Wronskian of $\vv f'(x)$. Let $S_0=\{x\in B_0:\varphi(x)=0\}$ and for $i=1,\dots,l-n$ let $S_i=\{x\in S_{i-1}:\varphi^{(i)}(x)=0\}$. By definition, $S_0\supset S_1\supset\dots\supset S_{l-n}$. By the choice of $B_0$, we must have that $S_{l-n}=\varnothing$. By Lemma~\ref{l2.14}, $S_{i-1}\setminus S_i$ is countable for every $1\le i\le l-n$. Hence $S_0=(S_0\setminus S_1)\cup\dots\cup(S_{n-l-1}\setminus S_{l-n})$ is countable and the proof is complete.
\end{proof}

In view of Proposition~\ref{prop2.13}, we can always apply Corollary~\ref{cor2.10} to nondegenerate curves with $l=n$. This gives the following statement.

\begin{corollary}\label{cor2.15}
Let $n\ge2$ be an integer, $\cC$ be a curve in $\R^n$, which is nondegenerate everywhere except possibly on a set of Hausdorff dimension $\le \frac{n+1}{\tau+1}-n+1$. Let $\tau\ge 1/n$ satisfy
\begin{equation}\label{f}
  \frac{n\tau-1}{\tau+1}\le \frac{3-2n\tau}{(2\tau+1)(2n-1)(n+1)}\,.
\end{equation}
Then
\begin{equation}\label{eqn2.20}
  \dim(\cC\cap\cS_n(\tau))=\frac{n+1}{\tau+1}-n+1\,.
\end{equation}
\end{corollary}

On letting $n\tau=1+\delta$, \eqref{f} transforms into
$$
  \frac{\delta}{n+1+\delta}\le \frac{1-2\delta}{(n+2+2\delta)(2n-1)(n+1)}
$$
or equivalently
\begin{equation}\label{f2}
\delta^2(4n^2+2n)+\delta(2n^3+5n^2+3n-1)-n-1\le0\,.
\end{equation}
Solving \eqref{f2} we get that
\begin{equation}\label{deltan}
0\le \delta\le \delta_n:=\frac{\sqrt{D_n}-B_n}{2A_n}
\end{equation}
where
\begin{align*}
  A_n & = 4n^2+2n\,,\\
  B_n & = 2n^3+5n^2+3n-1\,,\\
  D_n & = 4n^6+20n^5+37n^4+42n^3+23n^2+2n+1\,.
\end{align*}
By \eqref{f2}, we also have that
$$
\delta_n<\frac{1}{2n^2+5n}\,.
$$
This also means that the first term in \eqref{f2} is $<1$ for $n\ge3$ and therefore \eqref{f2} will hold whenever
\begin{equation}\label{f3}
\delta(2n^3+5n^2+3n-1)\le n\,.
\end{equation}
Also observe that $6n^2\ge 5n^2+3n-1$ for $n\ge3$. Hence \eqref{f3} is implied provided that $\delta(2n^2+6n)< 1$. Therefore
\begin{equation}\label{delta_n}
  \frac{1}{2n^2+6n} < \delta_n < \frac{1}{2n^2+5n}\,.
\end{equation}

\begin{corollary}[The spectrum of $\lambda_n$]
For every $n\ge3$
$$
\left[\frac1n,\frac1n+\frac{\delta_n}{n}\right]\;\subset\spec(\lambda_n)\,,
$$
where $\delta_n$ is given by \eqref{deltan} and can be estimated by \eqref{delta_n}.
\end{corollary}

\medskip

\section{Preliminaries}\label{sec3}

\subsection{Notation and conventions}

First let us agree on some notation that we will use throughout the rest of the paper.
By $\I_k$ we will denote the identity $k\times k$ matrix. Throughout $\|\cdot\|$ and $\|\cdot\|_\infty$ will denote the Euclidean and supremum norms on $\R^k$ respectively. Given $r>0$ and $\vv x\in\R^d$, by $B(\vv x,r)$ we will denote the Euclidean ball in $\R^d$ of radius $r$ centred at $\vv x$, and respectively, by $\cB(\vv x,r)$ we will denote the $\|\;\|_\infty$--ball of radius $r$ centred at $\vv x$, which for obvious reasons will be referred to as a hypercube.

We will use the Vinogradov and Bachmann--Landau notations: for
functions $f$ and positive-valued functions $g$, we write $f \ll g$ or $f = O(g)$ if
there exists a constant $C$ such that $|f| \le Cg$ pointwise. We will write $f \asymp g$ if $f \ll g$ and $g \ll f$.
Throughout $G = \SL(n+1, \R)$ and $\Gamma = \SL(n+1, \Z)$. Then the homogeneous space
$X_{n+1} := \SL(n+1, \R)/\SL(n+1, \Z)$ can be identified with the set of all unimodular
lattices in $\R^{n+1}$, where the coset $g\Gamma$ in $X_{n+1}$ corresponds to the lattice $g\Z^{n+1}$ in $\R^{n+1}$. Note that the column vectors of $g$ form a basis of $g\Z^{n+1}$.

\subsection{Preliminaries from the geometry of numbers}\label{sec-Preliminaries}

Given a lattice $\Lambda \in X_{n+1}$ and an integer $1\le i\le n+1$, let
\begin{equation}\label{eq:lambda}
\lambda_i(\Lambda):= \inf\Big\{\lambda>0: B(\vv 0,\lambda) \cap \Lambda \text{ contains $i$ linearly independent vectors} \Big\}\,.
\end{equation}
In other words, $\lambda_1(\Lambda)\le \dots\le\lambda_{n+1}(\Lambda)$ are the {\em successive minima}\/ of the closed unit ball $B(\vv 0,1)$ with respect to the lattice $\Lambda$.


Recall that, given a lattice $\Lambda\in X_{n+1}$, its {\em polar lattice} is defined as follows:
\begin{equation}\label{polar}
\Lambda^*=\{\vv a\in\R^{n+1}:\vv a\cdot\vv b\in\Z\text{ for every }\vv b\in\Lambda\}\,.
\end{equation}
The following lemma is well known, e.g. see \cite[Thm 21.5]{Gruber-2007}.

\begin{lemma}
Let $g \in G$. Then
$$
\left(g\Z^{n+1}\right)^\ast= (g^{\T})^{-1} \Z^{n+1},
$$
where $(g^{\T})^{-1}$ is the inverse of the transpose of $g$.
\end{lemma}

Given a convex body $\cC$ in $\R^{n+1}$ symmetric about $\vv0$, one defines the polar body
$$
\cC^*=\{\vv y\in\R^{n+1}:\vv x\cdot\vv y\le 1\text{ for all }\vv y\in\cC\}\,.
$$
It is readily seen that $B(\vv0,1)^*=B(\vv0,1)$. Then, the following theorem on successive minima of the polar lattice is a direct consequence of a more general result of Mahler, see \cite[Thm 23.2]{Gruber-2007}.

\begin{theorem}[Mahler, see {\cite[Thm 23.2]{Gruber-2007}}]\label{Mahler}
Let $\Lambda$ be any lattice in $\R^{n+1}$. Then for every $1\le i\le n+1$ we have that
$$
1\le \lambda_i(\Lambda)\lambda_{n+2-i}(\Lambda^*)\le (n+1)!^2\,.
$$
\end{theorem}

Given $k\in\N$, define the following square $k\times k$ matrix:
\[ \sigma_k = \begin{bmatrix}
  0 & 0 & \dots & 0 & 1 \\
  0 & 0 & \dots & 1 & 0 \\
  \vdots & \vdots & \ddots & \vdots & \vdots \\
  0 & 1 & \dots & 0 & 0 \\
  1 & 0 & \dots & 0 & 0
\end{bmatrix}. \]
In the case $k=n+1$ we will simply write $\sigma$ instead of $\sigma_{n+1}$. Note that $\sigma_k$ is an involution, that is $\sigma_k^{-1}=\sigma_k$. Also note that $\sigma_k$ acts on row-vectors on the right and column-vectors on the left by placing their coordinates in the reverse order. Furthermore, we have that $g\Z^{n+1}=g\sigma\Z^{n+1}$ and
$\sigma^{-1}(B(\vv0,\lambda))=B(\vv0,\lambda)$ for every $\lambda>0$. Therefore for every $g\in G$
\begin{equation}\label{eq2.4}
  \lambda_i(g\Z^{n+1})=\lambda_i(\sigma^{-1}g\sigma\Z^{n+1})\,.
\end{equation}

Given $g \in G$, we will define the {\em dual} of $g$, denoted by $g^{\ast}$, by
\begin{equation}\label{eq:dual-group}
  g^{\ast} := \sigma^{-1}(g^{\T})^{-1} \sigma\,.
\end{equation}
It is readily seen that the dual of the product of matrices equals the product of dual matrices, that is
\begin{equation}\label{eq:dual-group-prop}
  (g_1g_2)^{\ast} = g_1^*g_2^*\qquad\text{for any }g_1,g_2\in G\,.
\end{equation}
Further, in view of equation \eqref{eq2.4}, Theorem~\ref{Mahler} implies the following

\begin{lemma}
For any $g\in G$ and every $1\le i\le n+1$ we have that
$$
1\le \lambda_i(g\Z^{n+1})\lambda_{n+2-i}(g^*\Z^{n+1})\le (n+1)!^2\,.
$$
\end{lemma}

\subsection{A quantitative non-divergence estimate}

We will make use of a version of the quantitative non-divergence estimate on the space of lattices due to Bernik, Kleinbock and Margulis \cite[Theorem~6.2]{Bernik-Kleinbock-Margulis-01:MR1829381}. To be more precise, we will use a consequence of this non-divergence estimate appearing as Theorem~1.4 in \cite{Bernik-Kleinbock-Margulis-01:MR1829381}. Below we state it in a slightly simplified form which fully covers our needs. In what follows $\nabla$ stands for the gradient of a real-valued function.

\begin{theorem}[See {\cite[Theorem~1.4]{Bernik-Kleinbock-Margulis-01:MR1829381}}]\label{thm:KM}
Let $\cU\subset\R^d$ be open, $\vv x_0\in\cU$ and $\vv f:\cU\to\R^n$ be $l$--nondegenerate at $\vv x_0$. Then there exists a ball $B_0\subset\cU$ centred at $\vv x_0$ and a constant $E\ge1$ such that for any choice of
\begin{equation}\label{eq3.10}
0<\delta\le 1,\qquad T\ge1\quad\text{and}\quad K>0 \quad\text{satisfying}\quad \delta^{n}<K T^{n-1}
\end{equation}
the Lebesgue measure of the set
\begin{equation}\label{eq3.11}
\mathfrak{S}_{\vv f}(\delta,K,T):=\left\{\vv x\in B_0:\exists\;(a_0,\vv a)\in\Z\times\Z^n\;\text{such that}\; \left.
\begin{array}{l}
|a_0+\vv f(\vv x)\vv a^T|<\delta\\[1ex]
\|\nabla \vv f(\vv x)\vv a^T\|_\infty<K\\[1ex]
0<\|\vv a\|_\infty<T
\end{array}
\right.\right\}
\end{equation}
satisfies the inequality
$$
\cL_d\big(\mathfrak{S}_{\vv f}(\delta,K,T)\big)\;\le\;E \left(\delta K T^{n-1}\right)^{\frac{1}{d(2l-1)(n+1)}}\cL_d(B_0)\,.
$$
\end{theorem}

\medskip

\section{The generic and special parts}

\subsection{Dynamical reformulation}

Recall that
$$
\cR(\Delta;\ve,t)=\left\{(\vv p,q) \in \Z^{n+1}:0< q < e^t\text{ and }
  \exists\;\vv x \in \Delta\cap\cU\;\text{with}\;\vv f(\vv x)\in\cB\left(\frac{\vv p}{q},\frac{\ve}{e^t}\right)
\right\}.
$$
Our goal is to interpret the condition $\vv f(\vv x)\in\cB\left(\frac{\vv p}{q},\frac{\ve}{e^t}\right)$ in terms of properties of the action of $\gt$ on a certain lattice in $\R^{n+1}$.
With this goal in mind, given $\vv y = (y_1, \dots, y_n) \in \R^n$, define
\begin{equation}
  \label{eq:Uy}
  U(\vv y) := \begin{bmatrix}
  \I_n & \sigma_n^{-1} \vv y^T \\
   0 & 1
\end{bmatrix} = \begin{bmatrix}
  1  & & &  y_n \\
     &  \ddots & &  \vdots  \\
        & &  1 & y_1 \\
        & & &  1
\end{bmatrix}
\in G\,.
\end{equation}
Also given an $m\times d$ matrix $\Theta = [\theta_{i,j}]_{1\le i\le m,\;1\le j\le d} \in \R^{m \times d}$, let
\begin{equation}
  \label{eq:ZL}
  Z(\Theta) := \begin{bmatrix}
  \I_m & \sigma_m^{-1}\Theta\sigma_d & 0 \\
    0  & \I_d & 0 \\
    0 & 0 &  1
\end{bmatrix} = \begin{bmatrix} 1 & & &  \theta_{m, d} & \dots & \theta_{m, 1} & 0 \\
   &  \ddots & &  \vdots & \ddots & \vdots & \vdots \\
     & & 1 &  \theta_{1,d} & \dots & \theta_{1,1} & 0 \\
     & & & 1 & \dots & 0 & 0 \\
     & & & & \ddots & \vdots & \vdots  \\
     & & & & & 1 & 0 \\
    & & & & & &  1
\end{bmatrix} \in G.
\end{equation}
For each $t>0$ and $0<\ve<1$ define the following unimodular diagonal matrix
\begin{equation}
  \label{eq:gt}
  \gt:=\diag\Big\{\underbrace{\phi\ve^{-1},\dots,\phi\ve^{-1}}_n,\phi e^{-t}\Big\}\in G\,,
\end{equation}
where
\begin{equation}\label{eq:phit}
  \phi:=\Big(\ve^ne^t\Big)^{\frac{1}{n+1}}\,.
\end{equation}

Before moving on we state a couple of conjugation equations involving $\gt$.

\begin{lemma}\label{LemmaConjugate}
For any $t>0$, $\Theta\in \R^{m \times d}$ and $\vv y\in\R^n$ we have that
\begin{align}
  \label{eq:conju-g-u} \gt U(\vv y) \gt^{-1} &= U( e^t\ve^{-1}\vv y)\,, \\
  \label{eq:conju-g-z} \gt Z(\Theta) \gt^{-1} &= Z(\Theta)\,.
\end{align}
\end{lemma}

The proof is elementary and left to the reader.

\begin{lemma}\label{lm:psi-approximable-0}
Let $\vv y\in\R^n$. Then for any $t>0$, any $\Theta\in \R^{m \times d}$, if $\vv y\in\cB\left(\frac{\vv p}{q},\frac{\ve}{e^t}\right)$ for some $(\vv p,q)\in\Z^{n+1}$ with $0<q<e^t$ then
\begin{equation}\label{eqn3.7}
\|\gt Z(\Theta)U(\vv y) (-\vv p\sigma_n, q)^\T\|\le c_0 \phi\,,
\end{equation}
where
\begin{equation}\label{c_0}
c_0=\sqrt{n+1}\max_{1\le i\le m}(1+|\theta_{i,1}|+\dots+|\theta_{i,d}|)\,.
\end{equation}
\end{lemma}

\begin{proof}
To begin with, note that, by $\vv y\in\cB\left(\frac{\vv p}{q},\frac{\ve}{e^t}\right)$, we trivially have that
\begin{equation}\label{eq3.8}
  \| \gt  U(\vv y) (-\vv p\sigma_n, q)^\T \|_{\infty} < \phi\,.
\end{equation}
Then, using Lemma~\ref{LemmaConjugate} we get that
  \begin{align}
    \|\gt  Z(\Theta)U(\vv y) (-\vv p\sigma_n, q)^\T\|_\infty &\stackrel{\eqref{eq:conju-g-z}}{=} \|Z(\Theta) \gt  U(\vv y) (-\vv p\sigma_n, q)^\T\|_\infty \nonumber\\
    &\le \|Z(\Theta)\|_\infty\cdot\|\gt  u(\vv x) (-\vv p\sigma_n, q)^\T\|_\infty \nonumber\\
                                         &\stackrel{\eqref{eq3.8}}{\le} \|Z(\Theta)\|_\infty\cdot \phi\label{eq3.9}\,,
  \end{align}
where $\|Z(\Theta)\|_\infty$ is the operator norm of $Z(\Theta)$ as a linear transformation from $\R^{n+1}$ to itself equipped with the supremum norm.
As is well known $\|Z(\Theta)\|_\infty$ equals the maximum of $\ell_1$ norms of its rows, that is $\|Z(\Theta)\|_\infty=\max_{1\le i\le m}(1+|\theta_{i,1}|+\dots+|\theta_{i,d}|)$. Now, taking into account that $\|\vv a\|\le \sqrt{n+1}\|\vv a\|_\infty$ for any $\vv a\in\R^{n+1}$, we obtain \eqref{eqn3.7} immediately from \eqref{eq3.9}.
\end{proof}

Our next goal is to produce a similar statement when $\vv y=\vv f(\vv x)$, where $\vv f$ is as in \eqref{eq3.2} and subject to condition \eqref{lemma_translation_M}. To this end, for $\vv x = (x_1, \dots, x_d) \in \cU$, define
\begin{equation}\label{eq:vx}
  u(\vv x) := U(\vv f(\vv x))\,,
\end{equation}
where $U$ is given by \eqref{eq:Uy}, and let
\[ \J(\vv x) := \left[\frac{\partial f_{i} }{\partial x_j}(\vv x)\right]_{1\le i\le m,\;1\le j\le d} \in \R^{m \times d}  \]
denote the Jacobian of the map $\mv f(\vv x)=(f_1(\vv x),\dots,f_m(\vv x))$. Next, for $\vv x\in\cU$ define
\begin{equation}
  \label{eq:zx}
  z(\vv x) := Z(-\J(\vv x))\,,
\end{equation}
where $Z$ is given by \eqref{eq:ZL}, and finally let
\begin{equation}
  \label{eq:ux}
  \zu(\vv x) := z(\vv x) u(\vv x)\,.
\end{equation}
Explicitly, by the above definitions, we have that
\begin{equation}\label{zu}
\zu(\vv x)= \begin{bmatrix}
  \I_m & -\sigma_m^{-1}\J(\vv x)\sigma_d & \sigma_m^{-1}\mv h(\vv x)^T \\[1ex]
    0  & \I_d & \sigma_d^{-1}\vv x^T \\[1ex]
    0 & 0 &  1
\end{bmatrix}\,,
\end{equation}
where
$$
\mv h(\vv x)=(h_1(\vv x),\dots,h_m(\vv x))=\mv f(\vv x)-\J(\vv x)\vv x^T\,,
$$
that is
$$
h_i(\vv x)=f_i(\vv x)-\sum_{j=1}^dx_j\frac{\partial f_i(\vv x)}{\partial x_j}\quad(1\le i\le m)\,.
$$

\begin{lemma}\label{lm:psi-approximable}
Let $\vv x\in\cU$. If $\vv f(\vv x)\in\cB\left(\frac{\vv p}{q},\frac{\ve}{e^t}\right)$ for some $(\vv p,q)\in\Z^{n+1}$ with $0<q<e^t$ then
$$
\| \gt \zu(\vv x) (-\vv p\sigma_n, q)\| \le c_1 \phi,
$$
and in particular,
\begin{equation}\label{eqn4.10}
\lambda_1(\gt \zu(\vv x)\Z^{n+1})\le c_1 \phi\,,
\end{equation}
where
\begin{equation}\label{c_1}
c_1=\sqrt{n+1}(d+1)M\,.
\end{equation}
\end{lemma}

\begin{proof}
The proof is rather obvious and requires the following two observations. First, on setting $\Theta$ to be $-\J(\vv x)$ and $\vv y=\vv f(\vv x)$, by \eqref{eq:vx} and \eqref{eq:zx}, we get that
$$
Z(\Theta)U(\vv y)=z(\vv x)u(\vv x)=\zu(\vv x)\,.
$$
And second, the quantity
$\max_{1\le i\le m}(1+|\theta_{i,1}|+\dots+|\theta_{i,d}|)$ that appears in \eqref{c_0} is bounded by $(d+1)M$ in view of \eqref{lemma_translation_M}. The latter means that $c_0\le c_1$ and hence \eqref{eqn4.10} follows from \eqref{eqn3.7}.
\end{proof}

\subsection{The generic and special parts of a manifold}

Setting up the generic and special parts will require another diagonal action on $X_{n+1}$.
For each $t\in\N$ define the following diagonal matrix
\begin{equation}
\label{eq:dt}
b_t := \begin{bmatrix} e^{\frac{dt}{2(n+1)}} \I_{m} & & \\ & e^{-\frac{(m+1)t}{2(n+1)}}\I_d & \\ & & e^{\frac{dt}{2(n+1)}} \end{bmatrix} \in G\,.
\end{equation}
First define the `raw' set of the special part:
\begin{equation}\label{eq:minor0}
\fM_0(\ve,t):=\left\{\vv x\in \cU:\lambda_{n+1}\left( b_t \gt \zu({\vv x})\Z^{n+1}\right)> \phi\, e^{\frac{dt}{2(n+1)}}\right\}\,.
\end{equation}
Now define the \textbf{special part} as the following enlargement of $\fM_0(\ve,t)$ which will ensure the structural claim about $\fM(\ve,t)$ within Theorem~\ref{t1.3}\;:
\begin{equation}\label{eq:minor1}
\fM(\ve,t):=\bigcup_{\vv x\in \fM_0(\ve,t)}B\left(\vv x,\ve e^{-t/2}\right)\cap\cU\,.
\end{equation}
Naturally the \textbf{generic part} is the complement to $\fM(\ve,t)$\;:
\begin{equation}\label{Major}
\fM'(\ve, t):= \cU \setminus \fM(\ve,t).
\end{equation}

Before moving on we provide two further auxiliary statements. The first presents two conjugation equations involving $b_t$. For the rest of this paper, given $\vv y=(y_1,\dots,y_k)\in\R^k$ for some $1\le k<n$, with reference to \eqref{eq:Uy}, we define
$$
U(\vv y):=U(\tvv y)\qquad\text{with }\tvv y=(y_1,\dots,y_k,0,\dots,0)\in\R^n\,,
$$
while for any $A>0$
$$
U(O(A)):=U(\vv y)\qquad\text{for some }\vv y\in\R^n\quad\text{such that }\|\vv y\|\ll A\,.
$$

\begin{lemma}\label{LemmaConjugate2}
For any $t>0$, $\Theta\in \R^{m \times d}$ and $\vv x=(x_1,\dots,x_d)\in\R^d$ we have that
\begin{align}
  \label{eq:conju-b-u} b_t U\left(\vv x\right) b_{-t} &= U\left(e^{-t/2}\vv x\right)\,, \\[1ex]
  \label{eq:conj-b-z}  b_t Z(\Theta) b_{-t} &= Z(e^{t/2}\Theta)\,.
\end{align}
\end{lemma}

The proof of these equations is elementary and obtained by inspecting them one by one. The details are left to the reader.

\begin{lemma}\label{lm:conjugation}
For any $\vv x \in \cU$ and $\vv x' = (x'_1, \dots, x'_d) \in \R^d$ such that the line segment joining $\vv x$ and $\vv x+\vv x'$ is contained in $\cU$ we have that
\[
\zu(\vv x + \vv x') = Z(O(\|\vv x'\|)) U(O(\|\vv x'\|^2)) U\left(\vv x'\right) \zu(\vv x).
\]
\end{lemma}

The proof is readily obtained on using Taylor's expansion of $\vv f(\vv x')$ and \eqref{lemma_translation_M}. The details are left to the reader.

\medskip

\section{Proof of Theorem~\ref{t1.3}}

\subsection{Dealing with the special part}

The goal is to prove \eqref{eq1.4}, that is to give an explicit exponentially decaying bound for the measure of the special part $\fM(\ve, t)$, and to establish the structural claim about $\fM(\ve, t)$ that it can be written as a union of balls of radius $\ve e^{-t/2}$ of multiplicity $\le N_d$. Specifically, we prove the following statement.

\begin{proposition}\label{prop:minor-arcs}
Suppose $\cU\subset\R^d$ is open, $\vv x_0\in\cU$, $\vv f:\cU\to\R^n$ be given as in \eqref{eq3.2} and is $l$--nondegenerate at $\vv x_0$. Then there is a ball $B_0\subset\cU$ centred at $\vv x_0$ and constants $K_0,t_0>0$ depending on $\vv f$ and $B_0$ only with the following properties. For any $0<\ve\le 1$ and every $t\ge t_0$ we have that the set defined by \eqref{eq:minor1} satisfies
\[
\cL_d\big(\fM(\ve,t)\cap B_0\big) \le K_0 \left(\ve^{n}e^{\frac{3t}{2}}\right)^{-\frac{1}{d(2l-1)(n+1)}}\,.
\]
Furthermore $\fM(\ve,t)$ can be written as a union of balls in $\cU$ of radius $\ve e^{-t/2}$ of intersection multiplicity $\le N_d$.
\end{proposition}

\begin{proof}
By definition, for any $\vv x \in \fM_0(\ve,t)$, we have that
$$
\lambda_{n+1}( b_t \gt \zu(\vv x)\Z^{n+1}) > \phi e^{\frac{dt}{2(n+1)}}.
$$
By Theorem~\ref{Mahler} and property \eqref{eq:dual-group-prop}, we have that
\begin{equation}\label{eq5.6}
\lambda_1 ( b^\ast_t g^\ast_t \zu^\ast(\vv x)\Z^{n+1}) \le c_2 \phi^{-1} e^{-\frac{dt}{2(n+1)}}\,,
\end{equation}
where $c_2=(n+1)!^2$.
It is straightforward to see using \eqref{eq:dual-group}, \eqref{eq:gt}, \eqref{zu} and \eqref{eq:dt} that
\begin{equation}\label{eq:gt*}
  \gt^*:=\phi^{-1}\diag\Big\{e^{t},\underbrace{\ve,\dots,\ve}_n\Big\}\,,
\end{equation}
\begin{equation}\label{eq:dt*}
b_t^* := \begin{bmatrix}
 e^{-\frac{dt}{2(n+1)}} &&\\
& e^{\frac{(m+1)t}{2(n+1)}}\I_d & \\
&&e^{-\frac{dt}{2(n+1)}} \I_{m}
\end{bmatrix}\,,
\end{equation}
and
\begin{equation}\label{zu*}
\zu^\ast(\vv x)= \begin{bmatrix}
  1 & -\vv x& -\mv f(\vv x)\\[1ex]
    0  & \I_d & \J(\vv x) \\[1ex]
    0 & 0 &  \I_m
\end{bmatrix}\,.\hspace*{14ex}
\end{equation}
Therefore, by \eqref{eq5.6}, we get that for any $\vv x\in \fM_0(\ve,t)$ there exists $(a_0,\vv a)\in\Z\times\Z^n\setminus\{\vv0\}$ such that
\begin{align}
&|a_0+\vv f(\vv x)\vv a^T|<c_2e^{-t}\,,\label{eq5.11}\\[0ex]
&\|\nabla \vv f(\vv x)\vv a^T\|_\infty<c_2\ve^{-1}e^{-\frac{t}{2}}\,,\label{eq5.12}\\[0ex]
&\max\{|a_{d+1}|,\dots,|a_{n}|\}<c_2\ve^{-1}\,.\label{eq5.13}
\end{align}
Using \eqref{eq5.11}---\eqref{eq5.13}, \eqref{eq3.2}, \eqref{lemma_translation_M} and Taylor's expansion of the function $a_0+\vv f(\vv x)\vv a^T$ one has that
for every $\vv x'\in\fM(\ve,t)$
\begin{equation}\label{eq5.15}
|a_0+\vv f(\vv x')\vv a^T|<c_2e^{-t}+c_2de^{-t}+\tfrac12d^2mMc_2\ve e^{-t}\le c_3 e^{-t}
\end{equation}
where $c_3=c_2(1+n+n^3M)=(n+1)!^2(1+n+n^3M)$ depends on $n$ and $\vv f$ only.
Similarly, using \eqref{eq5.12}, \eqref{eq5.13}, \eqref{eq3.2}, \eqref{lemma_translation_M} and Taylor's expansion of the gradient $\nabla \vv f(\vv x)\vv a^T$ one has that
for every $\vv x'\in\fM(\ve,t)$
\begin{equation}\label{eq5.16}
  \|\nabla \vv f(\vv x')\vv a^T\|_\infty \le c_2\ve^{-1}e^{-\frac t2}+ dc_2e^{-\frac t2}\le c_3 \ve^{-1}e^{-\frac t2}\,.
\end{equation}
Also, by \eqref{eq5.12}, \eqref{eq5.13}, \eqref{eq3.2} and \eqref{lemma_translation_M} we also have that
\begin{equation}\label{eq5.14}
\max\{|a_{1}|,\dots,|a_{n}|\}<c_2mM\ve^{-1}\le c_3\ve^{-1}\,.
\end{equation}
Combining \eqref{eq5.15}---\eqref{eq5.14} gives that
\begin{equation}\label{eq5.17}
\fM(\ve,t) \subset \mathfrak{S}_{\vv f}(\delta,K,T)
\end{equation}
with
$$
\delta=c_3e^{-t},\quad K=c_3\ve^{-1}e^{-\frac{t}{2}},\quad T= c_3\ve^{-1}\,,
$$
where $\mathfrak{S}_{\vv f}(\delta,K,T)$ is defined by \eqref{eq3.11}.
It is readily seen that conditions \eqref{eq3.10} are satisfied for all $t$ such that $\delta=c_3e^{-t}\le 1$, that is $t\ge\log c_3=:t_0$.
Now fix any $\vv x_0\in\cU$ such that $\vv f$ is $l$--nondegenerate at $\vv x_0$ and let $B_0$ and $E$ be the ball and constant arising from Theorem~\ref{thm:KM}.
By Theorem~\ref{thm:KM} and \eqref{eq5.17}, we obtain that
$$
\cL_d \big(\fM(\ve,t)\cap B_0 \big) \le
E \left(c_3^{n+1}\ve^{-n}e^{-\frac{3t}{2}}\right)^{\frac{1}{d(2l-1)(n+1)}}\cL_d(B_0)\,,
$$
which gives the required bound with $K_0=E c_3^{\frac{1}{d(2l-1)}}\cL_d(B_0)$.

Finally, in view of the definition of $\fM(\ve,t)$, the `Furthermore' claim trivially follows from Besicovitch's covering theorem (see below) applied to the set $A=\fM(\ve,t)$ and $\cB$ being the collection of balls appearing in the right hand side of \eqref{eq:minor1}.
\end{proof}

\begin{theorem}[Besicovitch's covering theorem {\cite[Theorem 2.7]{Mat}}]
There is an integer $N_d$ depending only on $d$ with the following property: let $A$ be a
bounded subset of $\R^d$ and let $\cB$ be a family of nonempty open balls in $\R^d$ such
that each $x\in A$ is the center of some ball of $\cB$; then there exists a finite or
countable subfamily $\{B_i\}$ of $\cB$ covering $A$ of intersection multiplicity at most $N_d$, that is with $1_A \le \sum_i 1_{B_i} \le N_d$.
\end{theorem}

\subsection{Dealing with the generic part}

The goal is to give a sharp counting estimate for the number of rational points of bounded height near the generic part.
Indeed, the following statement we prove here completes the proof of Theorem~\ref{t1.3}.

\begin{proposition}\label{prop:major-arcs-counting}
Suppose $\cU\subset\R^d$ is open, $\vv f:\cU\to\R^n$ be a $C^2$ maps satisfying \eqref{eq3.2} and \eqref{lemma_translation_M}. Then for any $0<\ve\le1$, any ball $B\subset\cU$ and all sufficiently large $t$ we have that
\begin{equation}\label{eq1.5+}
N\Big(B\setminus\fM(\ve,t);\ve,t\Big)\le K_1\ve^{m}e^{(d+1)t} \cL_d(B)\,,
\end{equation}
where $K_1$ depends on $n$ and $\vv f$ only.
\end{proposition}

We will make use of the following trivial property: for any $\Delta_1,\Delta_2\subset\R^d$
\begin{equation}\label{eq:counting-rationals-prop}
  N(\Delta_1\cup\Delta_2;\ve, t) \le N(\Delta_1;\ve, t)+N(\Delta_2;\ve, t)\,.
\end{equation}
This allows us to reduce the proof of  Proposition~\ref{prop:major-arcs-counting} to considering domains of the form
$$
\Delta_t(\vv x_0):=\left\{\vv x\in\R^d:\|\vv x-\vv x_0\|_\infty\le \big(\ve e^{-t}\big)^{\frac12}\right\}\,,
$$
where $\vv x_0 \in \fM'(\ve,t)$. At the heart of the reduction is the following simple statement.

\begin{lemma}\label{lm:counting-reduction}
For all sufficiently large $t >0$ we have that
\[
N(B\setminus\fM(\ve,t);\ve, t) \le 2(\ve e^{-t})^{-\frac d2} \cL_d (B)  \max_{\vv x_0\in \fM'(\ve,t)\cap B} N(\Delta_t(\vv x_0)\cap B;\ve, t )\,.
\]
\end{lemma}

\begin{proof}
First of all note that $\ve e^{-t}\to0$ as $t\to\infty$, since $\ve\le1$ for all $t>0$. Therefore, for all sufficiently large $t$ the ball $B$ can be covered by $\le 2(\ve e^{-t})^{-d/2}\cL_d(B)$ hypercubes $\Delta$ of sidelength $(\ve e^{-t})^{1/2}$. Any of these hypercubes $\Delta$ that intersects $\fM'(\ve,t)\cap B$ can be covered by a hypercube $\Delta_t(\vv x_0)$ with $\vv x_0\in \fM'(\ve,t)\cap B\cap\Delta$. The collection of the sets $\Delta_t(\vv x_0)\cap B$ is thus a cover for $\fM'(\ve,t)\cap B= B\setminus \fM(\ve,t)$ of $2(\ve e^{-t})^{-d/2}\cL_d(B)$ elements. Applying \eqref{eq:counting-rationals-prop} completes the proof.
\end{proof}

In view of Lemma~\ref{lm:counting-reduction}, the following statement is all we need to complete the proof of Proposition~\ref{prop:major-arcs-counting}.

\begin{lemma}\label{lm:counting-estimate}
Let a ball $B\subset\cU$ be given. Then for all sufficiently large $t >0$ and all $\vv x_0 \in \fM'(\ve,t)\cap B$ we have that
\[
N(\Delta_t(\vv x_0)\cap B;\ve, t ) \ll \ve^ne^t (\ve e^{-t})^{-\frac d2}\,,
\]
where the implied constant depends on $n$ and $\vv f$ only.
\end{lemma}

\begin{proof}
Let us assume that $N(\Delta_t(\vv x_0)\cap B;\ve, t )\neq0$ as otherwise there is nothing to prove. Take any $(\vv p,q) \in \cR(\Delta_t(\vv x_0)\cap B;\ve, t)$. By definition, there exists $\vv x \in \Delta_t(\vv x_0)\cap B$ such that
$$
\left\|\vv f(\vv x)-\frac{\vv p}{q} \right\|_{\infty} <  \frac{\ve}{e^t}\,.
$$
By Lemma~\ref{lm:psi-approximable}, we have that
\begin{equation}\label{eqn5.14}
\| \gt \zu(\vv x) (-\vv p\sigma_n, q)\| \le c_1 \phi.
\end{equation}
Since $\vv x \in \Delta_t(\vv x_0)$, we have that
\[
\vv x_0 = \vv x + (\ve e^{-t})^{\frac12} \vv x'
\]
with $\|\vv x'\| \le 1$. Since $\vv x,\vv x_0\in B\subset\cU$, the line segment joining $\vv x_0$ and $\vv x$ is contained in $\cU$. Then, by Lemma \ref{lm:conjugation}, we have that
\[
\zu(\vv x_0) = Z\left(O\left((\ve e^{-t})^{\frac12}\right)\right) U\left(O(\ve e^{-t})\right) U\left((\ve e^{-t})^{\frac12}\vv x'\right) \zu(\vv x). \]
By \eqref{eq:conju-g-u} and \eqref{eq:conju-g-z}, we have
\[ \gt \zu(\vv x_0) = Z\left(O\left( (\ve e^{-t})^{\frac12}\right)\right) U(O(1))  U\left((\ve e^{-t})^{-\frac12}\vv x'\right) \gt \zu(\vv x). \]
Therefore,
\begin{align*}
  \gt & \zu(\vv x_0) (-\vv p\sigma_n, q)= \\[1ex]
    &=  Z\left(O\left( (\ve e^{-t})^{\frac12}\right)\right) U(O(1))  U\left((\ve e^{-t})^{-\frac12}\vv x'\right) \gt \zu(\vv x) (-\vv p\sigma_n, q).\end{align*}
Let us denote
\begin{equation}\label{eqn5.15}
\gt \zu(\vv x) (-\vv p\sigma_n, q) = \vv v = (v_n, \dots, v_1, v_0).
\end{equation}
Then, by the above calculation, we get that
\begin{align}
\nonumber  \gt \zu(\vv x_0) (-\vv p\sigma_n, q) &= Z\left(O\left( (\ve e^{-t})^{\frac12}\right)\right) U(O(1))  U\left((\ve e^{-t})^{-\frac12}\vv x'\right) \vv v \\[2ex]
                             &= Z\left(O\left( (\ve e^{-t})^{\frac12}\right)\right) U(O(1)) \vv v',\label{eqn5.16}
\end{align}
where
\[ \vv v' = \left(v_n, \dots, v_{d+1}, v_d + (\ve e^{-t})^{-\frac12} x'_d v_0, \dots, v_1 +  (\ve e^{-t})^{-\frac12} x'_1 v_0, v_0 \right).\]

By \eqref{eqn5.14} and \eqref{eqn5.15}, we have that $\|\vv v\| \le c_1 \phi$. Furthermore, since $0<q<e^t$, we get that $|v_0| = \phi e^{-t} q \le \phi$. Therefore, using $\|\vv x'\|\le 1$ we get that
\begin{equation}\label{eqn5.17}
\vv v' \in [c_1 \phi]^m \times [(c_1+1)\phi(\ve e^{-t})^{-\frac12}]^d \times [ \phi ],
\end{equation}
where $[a]$ denotes the closed interval $[-a,a]$. After considering the action of $Z(O( (\ve e^{-t})^{\frac12})) U(O(1))$ on $\vv v'$, we get from \eqref{eqn5.16} and \eqref{eqn5.17} that
\[ \gt \zu(\vv x_0) (-\vv p\sigma_n, q) \in [c_4 \phi]^{m} \times [c_4 \phi (\ve e^{-t})^{-\frac12} ]^d \times [ c_4 \phi ], \]
 for some constant $c_4 >0$ depending on $n$ and $\vv f$ only.
Then it is easy to verify that
\[ b_t \gt \zu(\vv x_0) (-\vv p\sigma_n, q) \in [c_4 \phi e^{h} ]^m \times [ c_4 \phi \ve^{-1/2} e^{h } ]^d \times [ c_4 \phi e^{h}], \]
where $h = dt/2(n+1)$.
Let us denote
$$\Omega = [c_4 \phi e^{h} ]^m \times [ c_4 \phi \ve^{-1/2} e^{h } ]^d \times [ c_4 \phi e^{h}].$$
Then
$$(-\vv p\sigma_n, q) \in \Omega \cap b_t \gt \zu(\vv x_0)\Z^{n+1} \subset (c_6 \Omega) \cap b_t \gt \zu(\vv x_0)\Z^{n+1} $$
for any $c_6 >1$. On the other hand, since $\vv x_0 \in \fM'(\ve,t)$, we have that
\[ \lambda_{n+1}( b_t \gt \zu(\vv x_0)\Z^{n+1}) \le \phi e^h.  \]
This implies that there exists a constant $c_6>1$ such that $c_6\Omega$ contains a full fundamental domain of
$b_t \gt \zu(\vv x_0)\Z^{n+1}$. Therefore,
\[ \# ((c_6\Omega )\cap  b_t \gt \zu(\vv x_0)\Z^{n+1}) \ll \cL_{n+1}(c_6 \Omega) \asymp \phi^{n+1} \ve^{-d/2} e^{(n+1)h},\]
which implies the desired estimate.
\end{proof}

\bigskip

\noindent{\it Acknowledgements.} LY wishes to thank the University of York for its hospitality during his visit in January 2020 when a significant part of this work was done. The authors are grateful to Sam Chow and the anonymous reviewer for their helpful comments.


\end{document}